\documentclass[11pt]{amsart}
\usepackage{amsmath}
\usepackage{mathtools}
 \usepackage[colorlinks=false, linktocpage=true, hidelinks=true]{hyperref}
\usepackage{hyperref}
\usepackage[left=1.2in, right=1.1in, top=1.1in, bottom=1.1in]{geometry}
\usepackage{tikz}
\usepackage{pgfmath-xfp}
\usepackage{pgfplots}
\usetikzlibrary{shadows}
\usetikzlibrary{shapes}
\usetikzlibrary{decorations}
\tikzset{               
    redarrows/.style={postaction={decorate},decoration={markings,mark=at position 0.1 with {\arrow[draw=black]{>}}},
           decoration={markings,mark=at position -0.4 with {\arrow[draw=black]{>}}},}}

\usepackage [utf8]{inputenc}
\usepackage {tikz}
\usepackage{tkz-euclide}
\usepackage[utf8]{inputenc}
\pgfplotsset{compat=1.18} 

\usepackage{amssymb}
\usepackage{amsmath}
\usepackage{amscd}
\usepackage{curves}
\usepackage{epsfig}
\usepackage{bbm}
\usepackage{geometry}
\usepackage{graphicx}
\usepackage{pstricks}
\usepackage{pstricks-add}
\usepackage{pst-grad}
\usepackage{pst-plot}
\usepackage{verbatim}
\usepackage{latexsym}
\usepackage{eucal}
\usepackage{amsfonts}
\usepackage{enumerate}

\numberwithin{equation}{section}
\newtheorem{theorem}{Theorem}[section]
\newtheorem{lemma}[theorem]{Lemma}
\newtheorem{prop}[theorem]{Proposition}

\newtheorem{definition}[theorem]{Definition}

\def \bpf {\begin{proof}}
\def \epf {\end{proof}}
\def \beq {\begin{equation}}
\def \eeq {\end{equation}}
\def \bsp{\begin{split}}
\def \esp{\end{split}}

\def \defn {:=}

\def \ovp {{\overline{p}}}

\def \wt {\widetilde}

\def \mca {{\mathcal A}}
\def \mcb {{\mathcal B}}

\def \mce {{\mathcal E}}

\def \mcf {{\mathcal F}}
\def \mcg {{\mathcal G}}
\def \mch {{\mathcal H}}

\def \mr {{\mathbb R}}
\def \mn {{\mathbb N}}
\def \ms {{\mathbb S}}

\def \mbz {{\mathbb Z}}
\def\ha {\frac{1}{2}}
\def \tha  {\frac32}

\def \oq {\frac{1}{4}}

\def \div {\operatorname{div}}

\def \mrn {{\mathbb R}^n}

\def \mcn {\mathcal{N}}

\def \msn {{\mathbb S}^n}

\def \ga {{\gamma}}
\def \eps {\varepsilon}

\def \vphi {\varphi}

\def \la {\lambda}   
   
\def \lan {\langle}   
\def \ran {\rangle}   
\def \del {\delta}   

\def \p {\partial}

\def \novt {\frac{n}{2}}

\def \beqq {\begin{equation}}

\def \eeqq {\end{equation}}

\def \ooh {\frac{1}{h}}

\def \ioh {\frac{i}{h}}

\numberwithin{equation}{section}

\begin{document}

\title[ Recovery of Null Forms] {The Recovery of Semilinear Potentials Satisfying Null Conditions From Scattering Data}
\author[Nathe]{Joel Nathe}

\author[S\'a Barreto]{Ant\^onio S\'a Barreto}
\address{Department of Mathematics, Purdue University \newline
\indent 150 North University Street, West Lafayette IN  47907, USA}
\email{jnathe@purdue.edu}
\email{sabarre@purdue.edu}
\thanks{ The second author is partly supported by the Simons Foundation grant \#848410. }

\begin{abstract} We construct oscillatory solutions of fully semilinear wave equations in Minkowski space  satisfying a null condition 
of the form
\[
\begin{split}
 \square u\defn &  \bigl(-\p_{x_0}^2 +\sum_{j=1}^n \p_{x_j}^2 \bigr)u=  q(x,u)\bigl((\p_{x_0}u)^2-|\nabla_{x'}u|^2\bigr), \\
& x=(x_0,x^{\prime}), \;\ x'=(x_1,\ldots, x_n)  \text{ and } x_0=t \text{ is the time variable, }
\end{split}
\]
on an interval $x_0\in [-T,T]$,  $T<\infty$ arbitrary, which consist of the superposition of a non-oscillatory background solution and  a single phase  train of highly oscillatory waves of wave length $h\ll1$ and amplitudes given by powers of  $h$; the waves interact with the nonlinearity and we measure the response $u(x_0,x')\bigr|_{x_0=T'}$ at a fixed time $x_0=T'<T$.

We show that  the coefficient of amplitude $h$ of the oscillatory part of the nonlinear geometric optics expansion of the solution  determines the light-ray transform of a vector field associated with $q(x,u)$, which determines $q(x,u)$ uniquely in the maximal region determined by the data.   Our methods also work for systems of semilinear wave equations satisfying null conditions, but in this paper we focus on the scalar case.
\end{abstract}
\maketitle

\tableofcontents

\section{Introduction}

 Numerous scientific experiments seek to  obtain information about a medium by probing it with waves and measuring its response.  These are  done for example in the study of laser interaction with crystals, oil prospection, ultrasound and other forms of medical imaging, \cite{EptSte1,EptSte2,KuvSmiCam,QiuFar,Tab,WanLiu} and it is often is the case that the medium is nonlinear, in which case the principle of superposition does not hold and waves interact with themselves and with each other.    On one hand, the nonlinearity adds serious difficulties  to the understanding of the propagation of waves, but on the other hand, the nonlinear interactions of the waves also reveal information about the medium, and  the question is how to judiciously extract it.

 We  consider an  inverse problem for semilinear wave equations satisfying a null condition, see \eqref{NLWE} below.  Examples of nonlinear wave equations satisfying null conditions are given by the wave maps equation  \cite{KlaMac,Tat,Tat1}. They also appear in connection with the study of  Yang Mills equations \cite{Chr} and Einstein equations \cite{LinRod}.

  We construct highly oscillatory solutions of scalar semilinear wave equations  \eqref{NLWE} below, for $x_0\in [-T,T]$, with $T$ arbitrary.  These are known as nonlinear geometric optics solutions.    We show that  the coefficient of amplitude $h$ of the oscillatory part of the nonlinear geometric optics expansion of the solution  restricted to $x_0=T'<T$ determines the light-ray transform of a vector field associated with $q(x,u)$, which determines $q(x,u)$ uniquely in the maximal region determined by the data. The construction also works for systems such as the ones studied in \cite{Sog}, but these will  be treated elsewhere.  Our focus is to present the methods in the scalar case.  The existence and structure of oscillatory solutions of  equations satisfying null conditions have been studied by several people, for example \cite{Cho,Cho1,HunLuk,Tou}, but as far as we know, the inverse problem for \eqref{NLWE} has not been treated.

  Nonlinear geometric optics  is a traditional subject  in physics, see for example the books \cite{Boy,Blo,BorWol,NewMol}. The rigorous mathematical theory has only been developed relatively recently, starting in the late 1980s, by several people, including Gu\`es, Joly, Majda, M\'etivier, Rauch \cite{Gue,JolMetRau,Maj1,Maj2}, see for example the books \cite{Met1,Rau} for a more thorough account.   The use of nonlinear geometric optics  methods  to recover semilinear potentials of the form $f(x,u)$ was first used  by Stefanov and the second author \cite{SaBSte,SaBSte1}.  More recently, Eptaminitakis and Stefanov have used such methods to study inverse problems for quasilinear equations \cite{EptSte1,EptSte2}.

  The rigorous study of nonlinear inverse problems for the wave equation has  gained a lot of attention in recent years beginning with with the work of Kurylev, Lassas and Uhlmann \cite{KurLasUhl} which showed that that the source to solution map for semilinear equations determines a Lorentzian manifold, modulo invariants.  While we aim to recover the nonlinearity,  the goal of \cite{KurLasUhl} is to recover information about the manifold. 
  
  This work \cite{KurLasUhl}  was  followed by many others, for example \cite{FeiLas,HinUhl,KurLasOksUhl,LasUhlWan,SaBUhlWan,TinWan,UhlZha,UhlZha1}.    The methods used in these papers are based on higher order linearization and propagation of singularities, and one has to filter parts of the solution, both in Sobolev regularity of the waves produced by the interactions of conormal waves, and in the powers of small parameters of the linearization.

  We will use the notation:
 \[
 \begin{split}
 &  x=(x_0,x^{\prime}), \;\ x^{\prime}=(x_1, x_2,\ldots, x_n),   \\
 & \nabla u=(\p_{x_0}u, \p_{x_1}u, \ldots, \p_{x_n}u) \text{ and } \nabla_{x^{\prime}}u=(\p_{x_1}u, \ldots, \p_{x_n}u),
 \end{split}
\]

  We consider  fully semilinear wave equations of the form
\beq\label{NLWE}
\begin{aligned}
&  \square u\defn  \, \bigl( -\p_{x_0}^2 +\sum_{j=1}^n \p_{x_j}^2\bigr)u = Q(x,u,\nabla u)=q(x,u)((\p_{x_0}u)^2-|\nabla_{x'} u|^2).
 \end{aligned}
\eeq
 Such nonlinear potentials $Q(x,u,\nabla u)$ are called null forms after \cite{Chr,Kla,KlaMac,Sog}.  In fact we can show that one can recover $f(x) q(x,u)$ for any $f\in C_0^\infty(\mr^{n+1})$ and hence may assume that $q(x,u)$  is compactly supported in the variable $x$.

 We will construct suitable oscillatory solutions $u$ of \eqref{NLWE}, and we show that the restriction of $u$ to $\{x_0=T\}$, for $T$ large enough, determines $q(x,u)$  provided it is compactly supported in $x$.  In general, the existence and uniqueness of solutions of semilinear equations of the form \eqref{NLWE} is only guaranteed for initial data with small energy \cite{Chr,Kla,KlaMac,Sog}, which is not sufficient to study the inverse problem; we need to construct bounded oscillatory solutions whose initial data does not have small energy.  The key point of the proof is to  construct a one-parameter family of approximate oscillatory solutions of \eqref{NLWE} in the sense of  \eqref{defuh} below, valid on arbitrary finite interval $[T_0,T]_{x_0}\times \mrn_{x^{\prime}}$.    We then use a result  of Gu\`es \cite{Gue} to show that there exists a unique one-parameter family of solutions of \eqref{NLWE} with the same initial data, valid in the same region $[T_0,T]_{x_0}\times \mrn_{x^{\prime}}$, which differs from the approximate solution by a large enough power of the wave length $h$, see Theorem  \ref{approx} below.  For the convenience of the reader, in Section \ref{PRGU},  we give a proof of Gu\`es' result  for this particular case.  We emphasize that the results of \cite{Gue} are more general and hold for quasilinear systems. It is also important to emphasize that for any  finite time interval there is an oscillatory solution, but the estimates depend on the size of the interval, so these solutions are not global.

The oscillatory solutions we construct are  such that the phases are still given by the linear eikonal equation for the wave operator, and the coefficients are determined by transport equations which involve the nonlinearity. This has been called the weakly nonlinear optics regime, see for example \cite{Met1,Rau}.

   We will adopt the following notation:
   \beq\label{notation}
   \begin{split}
   &  \text{ if } V_\pm=(\pm 1,\theta), \;\ \theta\in \ms^{n-1},  \text{ we denote } \wt V_\pm =(\mp 1,\theta), \\
 &\text{ and } \lan x,V_\pm\ran_{{}_M}= \mp x_0+  \lan x',\theta\ran= \mp x_0+ \sum_{j=1}^n x_j \theta_j,
 \end{split}
 \eeq
$ \lan x,V_\pm\ran_{{}_M}$ is the Minkowski scalar product.

  If $\vphi\in C^\infty(\mr)$ and $\theta\in \ms^{n-1},$ we denote
  \beq\label{defphi}
  \vphi_{{}_V}(x)\defn \vphi(\lan x,V\ran_{{}_M}), \;\  \vphi^\prime_{{}_V}(x)\defn \vphi^\prime(\lan x,V\ran_{{}_M}) \text{ and }  \vphi^{\prime\prime}_{{}_V}(x)\defn \vphi^{\prime\prime}(\lan x,V\ran_{{}_M}),
  \eeq
   and notice that
  \[
  \begin{split}
   &  \nabla \vphi_{{}_V}(x)  = \vphi_{{}_V}^{\prime}(x) \wt V, \;\  
   \p_{x_0}^2 \vphi_{{}_V}(x) = \vphi_{{}_V}^{\prime\prime}(x)
\text{ and }   \p_{x_j}^2 \vphi_{{}_V}(x)= \vphi_{{}_V}^{\prime\prime}(x) \theta_j^2, \;\ j=1, \ldots, n, 
  \end{split}
  \]
  and therefore
  \[
  \square \vphi_{{}_V}=(\p_{x_0}\vphi_{{}_V})^2-|\nabla_{x'} \vphi_{{}_V}|^2=0.
  \]
 We conclude that for $Q(x,u,\nabla u)$ as in \eqref{NLWE},
 \beq\label{equ0}
 \square \vphi_V(x)= Q\bigl(x,\vphi_V(x),\nabla \vphi_V(x)\bigr) = 0.
  \eeq

 We will think of $\vphi_V(x)$ as background solutions of \eqref{NLWE} and will construct solutions of \eqref{NLWE} which are oscillatory perturbations of $\vphi_V(x)$.  This is the main result of the paper:
 \begin{theorem}\label{main}  Let $q(x,u)$ be supported on $\{|x|\leq R\} \times \mr$. 
 Let  $\vphi , \chi\in C_0^\infty(\mr)$  be real valued, let $\omega, \theta\in \ms^{n-1}$,  $V=(\pm 1,\theta)$ and  $W=( -1,\omega)$. Let $\vphi_{{}_V}(x)$ and $\chi_{{}_W}(x)$  be defined as in \eqref{defphi}.  For $A,B\in \mr$,
 \[
\begin{split}
 u_{inc}(h,V,W,x)= \vphi_{{}_V}(x) +  h \chi_{{}_W}(x) \biggl( A \cos\biggl(\ioh \bigl( \lan x,W\ran_{{}_M}\bigr)\biggr)+ B \sin\biggl(\ioh \bigl( \lan x,W\ran_{{}_M}\bigr)\biggr)\biggr),
\end{split}
 \]
which satisfies \eqref{NLWE} if $x$ is not on the support of $\chi_{{}_W}$.  For any $N\in \mn$ and any interval $[T_0,T]_{x_0}$,  there exists a unique solution $u(h,V,W,x)$ of \eqref{NLWE}  such that 
\[
u(h,V,W,x)= u_{inc}(h,V,W,x) \text{ for } x_0\ll 0 \text{ outside the support of } q(x,u),
\]
which has an expansion of the form
\beq\label{defuh}
\begin{split}
& u_{N}(h,V,W,x)=  \vphi_{{}_V}(x) +  h \sum_{p=0}^N h^p \sum_{\pm m=0}^N  e^{\frac{im}{h} \lan x,W\ran_{{}_M}}  A_{m,p}(V,W,x)+ h^{N+1} \mce_N(h,V,W,x),\\
& \text{ where } A_{m,p} \in C^\infty, \;\ \overline{A_{m,p}}= A_{-m,p},  \text{ so  the expansion is real valued,} 
\end{split}
\eeq
and satisfy the following initial conditions when $x_0\ll 0$:
\beq\label{defcoeff}
\begin{split}
& A_{0,p}(V,W,x)=0, \\
&  A_{1,0}(V,W,x)= \ha \chi_{{}_W}(x)(A-i B), \;\    A_{m,0}(V,W,x)=0, \;\ \text{ if } m \ge 2,\\
 & A_{m,p}(V,W,x)=0, \text{ if } p\ge1.
\end{split}
 \eeq
The coefficient $A_{1,0}$ satisfies the transport equation
\beq\label{eqa1p0}
\begin{split}
& \bigl( \p_{x_0}- \omega\cdot \nabla_{x^{\prime}} \bigr) A_{1,0}= F(V,W,x) A_{1,0}, \text{ where } \\
& F(V,W,x)=\lan  q\bigl(x, \vphi_{{}_V}(x)\bigr)  \vphi_{{}_V}^\prime(x)   \wt V,\wt W \ran_M, \;\ \wt W=(1,\omega).
\end{split}
\eeq
Moreover, for any $\tau>0$, the error term satisfies
\[
\begin{split}
& \sup_{x_0\in [-\tau,\tau]}\| \mce_N\|_{{}_{L^2(\mr^n)}}\leq C(N,\tau), \;\  
\sup_{x_0\in [-\tau,\tau]}\| \mce_N\|_{{}_{L^\infty (\mr^n)}}\leq C(N,\tau), \\
& \sup_{x_0\in [-\tau,\tau]}\| D_x^\alpha \mce_N\|_{{}_{L^2(\mr^n)}}\leq C_\alpha(N,\tau) h^{1-|\alpha|}, \;\
\sup_{x_0\in [-\tau,\tau]}\| D_x^\alpha \mce_N\|_{{}_{L^2(\mr^n)}}\leq C_\alpha(N,\tau) h^{1-|\alpha|}, \;\ |\alpha|\geq 1.
\end{split}
\]
 \end{theorem}

 \subsection{The Inverse Problem}   
 
 The following is a consequence of Theorem \ref{main}

\begin{theorem}\label{inverse} Let $u(h,V,W,x)$ be the solution of \eqref{NLWE} constructed in Theorem \ref{main}.  If  $T'<T$ are large enough,  then $u(h,V,W,x)\bigr|_{\{x_0=T'\}}$  determines $q(x,u)$ uniquely.
\end{theorem}
\begin{proof}  If we set $x^\prime=y+ x_0\omega$ and $x_0=s,$ equation \eqref{eqa1p0} becomes
 \[
 \begin{split}
& \p_s A_{1,0}(V,W,s,y)= F(V,W,s, y+s \omega) A_{1,0}(V,W,s,y), \\
 & A_{1,0}(V,W,s,y)= \ha \chi_{{}_W}(s, y+s\omega) (A-iB) \text{ for } s\ll 0.
 \end{split}
 \]
 Therefore
 \[
 A_{1,0}(V,W,s,y)=\ha \chi_{{}_W}(s, y+s\omega) (A-iB) \exp\bigl( \int_{-\infty}^s F(V,W,\nu, y+\nu \omega) \ d\nu\bigr),
 \]
 and since $y=x'-s\omega$ and $s=x_0$,
 \[
  A_{1,0}(V,W,x_0,x^\prime)=\ha \chi_{{}_W}(x) (A-iB) \exp\bigl( \int_{-\infty}^{x_0} F(V,W,x_0+\nu, x^\prime+\nu \omega) \ d\nu\bigr).
  \]
  
 Since $q$ is compactly supported, if $T'<T$ is large enough,
 \[
 \begin{split}
 & A_{1,0}(V,W,T^\prime,x^\prime)=\ha \chi_{{}_W}(x) (A-iB) \exp\bigl( \int_{-\infty}^\infty F(V,W,T'+\nu, x^\prime+\nu \omega) \ d\nu\bigr)= \\
 & \ha \chi_{{}_W}(x) (A-iB) \exp\bigl( \int_{-\infty}^\infty \lan \mcf(V,T'+\nu, x^\prime+\nu \omega), \wt W\ran  \ d\nu\bigr).
 \end{split}
   \]
 where $\lan,\ran$ is the Euclidean dot product.  By  shifting the time variable by $-T'$ we may just take $T'=0.$  The integral
  \[
 L_1(\mcf)(x^\prime,V,W)= \int_\mr \lan \mcf(V, \nu, x^\prime  + \nu\omega), \wt W\ran \ d\nu, \;\ \wt W=(1,\omega), \;\ \omega \in \ms^{n-1}
 \]
 is defined to be the future light-ray transform of the vector field $\mcf$ defined in \eqref{eqa1p0},  see for example Section III.4 of \cite{SteUhl}.

We conclude from Theorem \ref{main}  that  the solution $u(h,V,W,0, x')$ given by \eqref{defuh}  satisfies
 \[
 \begin{split}
 u(h,V,0,x')= &  \vphi_{{}_V}(x)+ h A_{0,0}(x) + h A\chi_{{}_W}(x) e^{L_1(\mcf)(x,V,W)} \cos\bigl(\ooh \lan x,W\ran_{{}_M}\bigr) + \\
&  h B \chi_{{}_W}(x) e^{L_1(\mcf)(x,V,W)} \sin\bigl(\ooh \lan x,W\ran_{{}_M}\bigr) + O(h^2).
\end{split}
 \] 
 and so we deduce that the coefficient of amplitude $h$ of the oscillatory part of $u(h,V,0,x^{\prime\prime})$ determines the light ray transform of $\mcf(x,V)$.  If  $q_1(x,u)$ and $q_2(x,u)$ are compactly supported in $x$ and the corresponding solutions defined by \eqref{defuh} agree at time $x_0=T'$, for $T'<T$ large enough, then $L_1(\mcf_{q_1}- \mcf_{q_2})=0$, where $\mcf_\bullet$ is the vector field defined in \eqref{eqa1p0} corresponding to $\bullet=q_1,q_2.$  But according to Proposition III.4.3 of \cite{SteUhl}, see also \cite{Sia},  if $q(x,u)$ is compactly supported in $x$, $L_1(\mcf)=0$ if and only if the exterior derivative $d \mcf'=0$, where $\mcf'$ is the one-form dual of $\mcf.$  The result follows from the following:
\begin{prop}\label{etxt-der}  Let $q(x,u)\in C^\infty$, be compactly supported in $x$.  If 
 $\vphi\in C_0^\infty(\mr)$  and $V=(\pm 1,\theta),$ $\theta \in \msn$,  let $\eta_{{}_{\vphi,V}}$ be the one form defined by
\[
\eta_{{}_{\vphi,V}}= \sum_{j=0}^n q(x,\vphi_{{}_V}(x)) \vphi^\prime_{{}_V}(x) V_j \ dx_j, 
\]
then  $d\eta_{{}_{\vphi,V}}=0$ for all $\vphi$ and all $V$  if and only if $q=0.$
\end{prop}
\begin{proof} We have that
\[
\begin{split}
& d \eta_{{}_{\vphi,V}}= 
\sum_{j=0}^n \sum_{m<j}   \biggl( \p_{x_m} \bigl( q(x, \vphi_{{}_V}(x)) \vphi_{{}_V}^\prime(x) \wt V_j\bigr)-  \p_{x_j} \bigl( q(x, \vphi_{{}_V}(x)) \vphi_{{}_V}^\prime(x) \wt V_m\bigr)\biggr) dx_m \wedge dx_j,
\end{split}
\]
and so, $d \eta_{{}_{\vphi,V}}= 0$ if and only if
\[
 \p_{x_m} \bigl( q(x, \vphi_V(x)) \vphi_{{}_V}^\prime(x) \wt V_j\bigr)-  \p_{x_j} \bigl( q(x, \vphi_V(x)) \vphi_{{}_V}^\prime(x) \wt V_m\bigr)=0, \text{ for all } m<j 
 \]

 But
\[
\begin{split}
 & \p_{x_m} \bigl( q(x, \vphi_V(x)) \vphi_{{}_V}^\prime(x) \wt V_j\bigr)=
 ( \p_{x_m}  q) (x, \vphi_V(x)) \vphi_{{}_V}^\prime(x) \wt V_j+ \\
&  (\p_u  q) (x, \vphi_V(x)) \bigl(\vphi_{{}_V}^\prime(x)\bigr)^2 \wt V_j \wt V_m+ 
  q(x, \vphi_V(x)) \vphi_{{}_V}^{\prime\prime}(x) \wt V_j \wt V_m.
  \end{split}
  \]
 So we conclude that $d \eta_{{}_{\vphi,V}}= 0$ if and only if
 \beq\label{GR}
 \begin{split}
&  \biggl( (\p_{x_m}q)(x,\vphi_{{}_V}(x)) \wt V_j- (\p_{x_j} q)(x,\vphi_{{}_V}(x)) \wt V_m \biggr)\vphi_{{}_V}^\prime(x)=0, \\
&  \text{ for all } V=(\pm 1,\theta) \text{ and } m<j \text{ and all } \vphi\in C_0^\infty(\mr).
\end{split}
 \eeq
 If $m=0,$ $\wt V_0=\pm 1$, for $j>0$ fixed, we can choose $\wt V$ such that  $\wt V_j=0$ and so
 \[
 (\p_{x_j} q)(x,\vphi_{{}_V}(x)) \vphi_{{}_V}^\prime(x)=0,  \text{ for all } \vphi\in C_0^\infty(\mr).
 \]
 For this choice of $V$, for any $u\in \mr$, and each $x$ we can choose a $\vphi$ such that $\vphi_{{}_V}^\prime(x)=1$ and $\vphi_{{}_V}(x)=u$. This implies that $\p_{x_j} q(x,u)=0$, for all $j>0.$ Plugging this back into \eqref{GR}, it  implies that  $(\p_{x_0} q)(x,\vphi_{{}_V}(x)) \vphi_{{}_V}^\prime(x)=0$ for all $\vphi$. So $\p_{x_0} q(x,u)=0$ and since $q(x,u)$ is compactly supported, $q(x,u)=0$. This ends the proof of the Proposition.
\end{proof}

This proves uniqueness. One can actually invert this light ray transform,  but it is unstable and we refer the reader to Section III.4 of \cite{SteUhl}, see also \cite{Sia,Ste}.
In the case  $q(x,u)=q(x',u)$, i.e. when  $q$ does not depend on $x_0$, the light ray transform becomes the $X$-ray transform, which is stably  invertible, see again \cite{Sia,Ste,SteUhl}.

\end{proof}
 \section{Approximate Solutions of Equation \eqref{NLWE}}

We will construct approximate solutions of \eqref{NLWE} of the following form:
 
 \begin{theorem}\label{mainA}  Let $q(x,u)$ be supported in $\{|x|\leq R\} \times \mr$.  Let  $\chi, \vphi \in C_0^\infty(\mr)$,  be real valued, let $\theta, \omega\in \ms^{n-1},$ $\omega\not=\theta$  and,   $V=( \pm 1,\theta)$ and  $W=( 1,\omega)$.  For $A,B\in \mr$ let
  \[
 u_{inc}(h,V,W,x)= \vphi_{{}_V}(x) + h \chi_{{}_W}(x) \biggl( A \cos\biggl(\ioh \bigl( \lan x,W\ran_{{}_M}\bigr)\biggr)+ B \sin\biggl(\ioh \bigl( \lan x,W\ran_{{}_M}\bigr)\biggr)\biggr),
 \]
 then for any $N\in \mn,$  there exists an approximate solution $u_{N}(h,V,W,x)$ of \eqref{NLWE}  in the sense that
\beq\label{estGN}
\begin{split}
& \square u_{N}-  Q(x,u_{N},\nabla u_{N}) = h^{N+1} G_N(h,V,W, x), \\
&  \text{ such that }   u_N(h,V,W,x)= u_{inc}(h,V,W,x)    \text{ for }  x_0\ll 0, \text{ outside the support of } q, \\
&G_N\in C^\infty  \text{ and for any } k \in \mn, \;\ T_0,T \in \mr,  \;\ G_N \text{ satisfies } \\
& \sup_{x_0\in [T_0,T]}  || \p_{x_0}^ r D_{x'}^\alpha G_N(h,V,W,x)||_{L^2(\mrn_{x^{\prime}})} \leq C(N,k,T_0,T)  h^{-k}, \;\ r+ |\alpha|=k, \\
& \sup_{x_0\in [T_0,T]}  || D_x^\alpha G_N(h,V,W,x)||_{L^\infty(\mrn_{x^{\prime}})} \leq C(N,k,T_0,T)  h^{-k}.
\end{split}
\eeq
 Moreover, $u_N(h,V,W,x)$ has an expansion of the form
\beq\label{defuh1}
\begin{split}
& u_{N}(h,V,W,x)=  \vphi_{{}_V}(x)
+ h \sum_{p=0}^N h^p v_p(h,V,W,x), \;\ u_0(V,x)=\vphi_{{}_V}(x),  \text{ such that } \\
& v_p(h,V,W, x) = \sum_{\pm m=0}^N  v_{m,p}(V,W,x), \;\ v_{m,p}(V,W,x)= e^{\frac{im}{h} \lan x,W\ran_{{}_M}}  A_{m,p}(V,W,x),
\end{split}
 \eeq
where the coefficients satisfy the initial conditions \eqref{defcoeff} and $A_{1,0}$ satisfies the transport equation \eqref{eqa1p0}. Moreover, $u_N$ is unique in the sense that if $\wt u_N(h,V,W,x)$ satisfies \eqref{estGN}, then 
\beq\label{uniqueness}
\sup_{x_0\in [T_0,T]} \|u_N-\wt u_N\|_{H^k(\mrn)}=O(h^{N+1-k}), \;\ k\in \mn.
\eeq
The same construction works for $W=(-1,\omega)$ with only minor changes. 
 \end{theorem}
 
 \begin{proof}  To find the coefficients of the expansion of $u_N$ given by 
 the first equation of \eqref{defuh} we first notice that
 \[
 \square \bigl(e^{i \frac{m}{h} \lan x,W\ran_{{}_M}} v\bigr)=
 e^{i \frac{m}{h} \lan x,W\ran_{{}_M}}\biggl( \frac{2im}{h} T v+\square v\biggr), \text{ where }
 T= \p_{x_0}- \omega\cdot \nabla_{x'}
 \]
  Let us denote $u_0(V,x)=\vphi_{{}_V}(x)$, and since $\square u_0=0,$  if $u_N$ is given by \eqref{defuh}, then
 \beq\label{LHS}
 \begin{split}
& \square u_N=  \sum_{p=0}^N h^{p+1} \square A_{0,p} + \sum_{p=0}^N 
 \sum_{\pm m=1}^N e^{\frac{im}{h} \lan x,W\ran_{{}_M}}\bigl( 2im h^p T A_{m,p} + h^{p+1} \square A_{m,p}\bigr)= \\
& \sum_{\pm m=1}^N  e^{\frac{im}{h} \lan x,W\ran_{{}_M}}\bigl(2im T A_{m,0}\bigr)+ 
 \sum_{p=1}^N h^p\bigl( \sum_{\pm m=1}^N  e^{\frac{im}{h} \lan x,W\ran_{{}_M}} 
\bigl( 2i m TA_{m,p}+ \square A_{m,p-1}\bigr) \bigr) + h \square A_{0,p}.
\end{split}
\eeq
 Let us denote
 \[
 \begin{split}
& \lan V,W\ran_{{}_M}= \sum_{j,k=0}^n Q_{j,k} V_j W_k, \;\ Q_{0,0}=-1, \;\ Q_{j,0}=Q_{0,j}=0 \;\ j \geq 1 \text{ and }
 Q_{j,k}=\del_{j,k}, \;\ j,k \geq 1, \\
 & \text{ and define } Q_{j,k}(x,u)=: q(x,u) Q_{j,k}.
 \end{split}
 \]
 Since  $\lan \nabla u_0(x),\nabla u_0(x)\ran_{{}_M} =0$, we have
 \[
 \begin{split}
&  Q(x,u_N, \nabla_x u_N)=  \sum_{j,k=0}^n Q_{j,k}(x,u_N) \bigl(\p_{x_j}u_0+ h \sum_{p=0}^N h^p \p_{x_j} v_p\bigr)\bigl(\p_{x_k}u_0+ h\sum_{p=0}^N  h^p \p_{x_k} v_p\bigr)= \\
& \sum_{j,k=0}^n\sum_{p=0}^N h^{p+1} Q_{jk}(x,u_N)\bigl(\p_{x_j} u_0 \p_{x_k} v_p+ \p_{x_k} u_0 \p_{x_j} v_p\bigr)+ \\
& \sum_{j,k=0}^n\sum_{p=0}^{2N} h^{p+2}\sum_{\mu+\nu=p}  Q_{jk}(x,u_N)\bigl(\p_{x_j} v_\mu \p_{x_k} v_\nu+ \p_{x_k} v_\nu \p_{x_j} v_\mu\bigr)
 \end{split}
 \]
 Notice that
\[
\begin{split}
&  h Q_{jk}(x,u_N)\bigl(\p_{x_j} u_0 \p_{x_k} v_p+  \p_{x_k} u_0 \p_{x_j} v_p\bigr)= \\
& \sum_{\pm m=1}^N e^{\frac{im}{h} \lan x,W\ran_{{}_M}} im A_{m,p} Q_{jk}(x,u_N)\bigl( (\p_{x_j} u_0) \wt W_k + (\p_{x_k} u_0) \wt W_j\bigr)+ \\
& h \sum_{\pm m=0}^N e^{\frac{im}{h} \lan x,W\ran_{{}_M}}  Q_{jk}(x,u_N)\bigl( (\p_{x_j} u_0) (\p_{x_k} A_{m,p})+ (\p_{x_k} u_0)( \p_{x_j} A_{m,p}).
\end{split}
\]
Similarly, using that $\lan \wt W, \wt W\ran_{{}_M}=0,$ we find that
\[
\begin{split}
& h  \sum_{j,k=0}^n Q_{jk}(x,u_N)\bigl(\p_{x_j} v_\mu \p_{x_k} v_\nu+  \p_{x_k} v_\nu \p_{x_j} v_\mu\bigr)= \\
& \sum_{j,k=0}^n \sum_{\pm l=0}^N \sum_{\pm m=0}^N e^{\frac{i(m+l)}{h} \lan x,W\ran_{{}_M}} 
 Q_{jk}(x,u_N)\biggl( im  A_{m,\mu} \wt W_j\p_{x_k} A_{l,\nu}+  il  A_{l,\nu} \wt W_k\p_{x_j} A_{m,\mu}\biggr)+ \\
&h \sum_{j,k=0}^n \sum_{\pm l=0}^N \sum_{\pm m=0}^N e^{\frac{i(m+l)}{h} \lan x,W\ran_{{}_M}}  Q_{jk}(x,u_N)\biggl( \p_{x_j} A_{m,\mu}\p_{x_k} A_{l,\nu} + \p_{x_k} A_{m,\mu}\p_{x_j} A_{l,\nu} \biggr).
\end{split}
\]
So we conclude that
\beq\label{RHS}
\begin{split}
& Q(x,u_N,\nabla u_N)= 
\sum_{j,k=0}^n\sum_{p=0}^N h^p  \sum_{\pm m=1}^N e^{\frac{im}{h} \lan x,W\ran_{{}_M}} im A_{m,p} Q_{jk}(x,u_N)\biggl( (\p_{x_j} u_0) \wt W_k + (\p_{x_k} u_0) \wt W_j\biggr)+ \\
&h \sum_{j,k=0}^n \sum_{p=0}^N  h^p \sum_{\pm m=0}^N e^{\frac{im}{h} \lan x,W\ran_{{}_M}}  Q_{jk}(x,u_N)\biggl( (\p_{x_j} u_0) (\p_{x_k} A_{m,p})+ (\p_{x_k} u_0)( \p_{x_j} A_{m,p}) \biggr) + \\
 & h \sum_{j,k=0}^n \sum_{p=0}^{2N} h^p \sum_{\mu+\nu=p} \ \sum_{\pm l=0, \pm m=0}^N e^{\frac{i(m+l)}{h} \lan x,W\ran_{{}_M}}  Q_{jk}(x,u_N)\biggl( im  A_{m,\mu} \wt W_j\p_{x_k} A_{l,\nu}+  il   A_{l,\nu} \wt W_k\p_{x_j} A_{m,\mu}\biggr)+ \\
& h^2\sum_{j,k=0}^n  \sum_{p=0}^{2N} h^p \sum_{\mu+\nu=p} \ \sum_{\pm l=0, \pm m=0}^N  e^{\frac{i(m+l)}{h} \lan x,W\ran_{{}_M}} Q_{jk}(x,u_N) \biggl( \p_{x_j} A_{m,\mu}\p_{x_k} A_{l,\nu} +  \p_{x_k} A_{m,\mu}\p_{x_j} A_{l,\nu} \biggr).
\end{split}
\eeq

We recall that   $u_N= u_0+ h v$, with  $v=\sum_{p=0}^N h^p v_p$, and so it follows from \eqref{LHS} and \eqref{RHS} that $u_N$ satisfies \eqref{estGN} if and only if  the following identities hold for the terms with $p=0$:
\beq\label{level1}
\begin{split}
& T A_{m,0}=-\ha \sum_{j,k=0}^n Q_{j,k}(x, u_0) \biggl( (\p_{x_j} u_0) \wt W_k + (\p_{x_k} u_0) \wt W_j\biggr) A_{m,0},\;\ m\geq 1, \\
& A_{1,0}= \ha \chi\bigl(\lan x,W\ran_{{}_M} \bigr) (A-i B), \text{ and }  \;\ A_{m,0}=0 \text{ for } x_0\ll0, \;\ m\geq 1,
\end{split}
\eeq
and 
\beq\label{level11}
\begin{split}
\square A_{0,0}= & \sum_{j,k=0}^n Q_{jk}(x,u_0)\biggl( (\p_{x_j} u_0) (\p_{x_k} A_{0,0})+ (\p_{x_k} u_0) (\p_{x_j} A_{0,0})\biggl), \\
& A_{0,0}=0 \text{ for } x_0\ll0.
\end{split}
\eeq
Therefore, this determines $A_{m,0}$ for $m\in \mbz$ and $|m|\leq N$ and hence it determines $v_0$. Notice that, by finite speed of propagation,  for $x_0\in [T_0,T]$,    $A_{m,0}$ is compactly supported in $x'$.

Now, suppose we have determined $A_{m,j}$, for $j\leq p$ and  $m\in \mbz$, $|m|\leq N,$  so that \eqref{estGN} holds, and hence we have determined $v_j,$ with $j\leq p,$  and we  want to determine $A_{m,p+1}$, $m\in \mbz$, $|m|\leq N.$

We need to take into account the powers of $h$ coming from $Q_{j,k}(x, u_N)$ and its Taylor's expansion to order one gives
 \[
 \begin{split}
 & Q_{j,k}(x,u_N)= Q_{j,k}(x,u_0+ h v)=  Q_{j,k}(x, u_0)+ h(\p_u Q_{j,k})(x,u_0) v_0+ h^2  \wt Q_{j,k}(h,x,V,W)
 \end{split}
 \]

 Notice that, the third and fourth terms of \eqref{RHS} will only have terms with  $\mu+\nu=p\geq 0$, and those involve previously computed terms.   For the terms $m\not=0$ we obtain
the following equation
\[
\begin{split}
-2 T A_{m,p+1}= &  Q_{jk}(x,u_0)\biggl( (\p_{x_j} u_0) \wt W_k + (\p_{x_k} u_0) \wt W_j\biggr)A_{m,p+1} + \mce_{m,p}( u_0, v_0, \ldots v_p), \\
& A_{m,p+1}=0 \;\ x_0\ll 0,
\end{split}
\]
and for $m=0$ 
\[
\begin{split}
\square A_{0,p+1}= & Q_{jk}(x,u_0)\biggl( (\p_{x_j} u_0) (\p_{x_k} A_{0,p+1})+ (\p_{x_k} u_0) (\p_{x_j} A_{0,p+1})\biggl)+ \wt \mce_{p}( u_0, v_0, \ldots v_p), \\ 
& A_{0,p+1}=0 \text{ for } x_0\ll0.
\end{split}
\]
Notice that it follows from \eqref{RHS} that $\wt \mce_{p}( u_0, v_0, \ldots v_p)$ and $ \mce_{m,p}( u_0, v_0, \ldots v_p)$ are compactly supported in $x'$ on any finite interval $x_0\in [T_0,T]$.  
These equations can be solved and this establishes the existence of the asymptotic expansion.

To prove \eqref{estGN} for the error term, we just need to observe that it follows from the discussion that 
for any finite interval $[T_0,T]$,
\[
 G_N(h,V,x) =\sum_{m=0}^N e^{\ioh \lan x,W\ran_{{}_M}} \mcg_m(h, V,W,x), \text{ with }  \mcg_m(h, V,W,x) \in C^\infty([T_0,T]; C_0^\infty(\mrn_{x^\prime})),
 \]
and hence the estimate \eqref{estGN} is obvious. The uniqueness of the solution in the sense of \eqref{uniqueness} also follows from the proof, since the coefficients of the expansion will satisfy the same transport and wave equations.
\end{proof}

 \section{From Approximate Solutions to Solutions of  \eqref{NLWE}}\label{PRGU}

In general one can only find long time, or even local, solutions of  \eqref{NLWE}  for small energy data, see for example \cite{Kla,KlaMac,Sog}, which would not suffice to solve the inverse problem.  The existence and uniqueness of solutions of \eqref{NLWE} with oscillatory initial data on an interval of time which does not depend on $h$ was established in \cite{Del}. However, this result still does not give control on the size of the time interval.     The key point here is that the existence of an approximate solution gives quite a bit of information.  Gu\`es \cite{Gue} has shown that the existence of a (one parameter family of) suitable approximate solutions of \eqref{NLWE} in $\Omega=[T_0,T]_{x_0} \times \mrn_{x^{\prime}}$ guarantees the existence of  unique (in the sense of Theorem \ref{approx} below) one parameter family of solutions of \eqref{NLWE} in the same region $\Omega$ with the same initial data.  One could write \eqref{NLWE} as a nonlinear system and verify that the hypotheses of Theorem 1.1 of \cite{Gue} are satisfied and  obtain the result we need.   However, for the convenience of the reader, and without claiming any originality,  we give a proof of Theorem 1.1 of \cite{Gue}  in this particular case.  We note that the results of \cite{Gue} are for quasilinear systems and necessarily more technical.  The results of \cite{Gue} have been used in \cite{EptSte1,EptSte2} to study inverse problems.  

    As usual, for $m\in \mr$ and $1\leq p\leq \infty$, $W^{m,p}$ will denote the standard $L^p$ based Sobolev spaces, and $H^m= W^{m,2}$. As in \cite{Gue} we use semiclassical Sobolev spaces, see \cite{Zwo} for more details on these spaces.     For $\Omega=[T_0,T]_{x_0} \times \mrn_{x^{\prime}}$, $\rho>0$, $\eps>0$ and $m\in \mn$,  we say that
\beq\label{GspB}
\begin{split}
& w_h \in \mcb_{\rho,\eps}^m(\Omega) \text{ if }  w_h, D w_h  \in C^0\bigl( [T_0,T], H^m( \mrn_{x^{\prime}})), \text{ and } \\ 
& \sup_{x_0 \in [T_0,T]} \| (h D_{x^{\prime}})^k  w_h(x_0, \bullet)\|_{{}_{H^1(\mrn)}} +
\sup_{x_0 \in [T_0,T]} \|D_t (h D_{x'}) k w_h(x_0, \bullet)\|_{{}_{L^2(\mrn)}}  \leq \rho, \\
& \text{ for all } 0\leq k\leq m, \text{ and } h\in (0,\eps] .
\end{split}
\eeq

Similarly, we say that
\beq\label{GspA}
\begin{split}
& w_h \in \mca_{\rho,\eps}^m(\Omega) \text{ if }  w_h, D w_h  \in C^0\bigl( [T_0,T], W^{m,\infty} ( \mrn_{x^{\prime}})), \text{ and for all }  1\leq k\leq m \text{ and } h\in (0,\eps], \\ 
& \sup_{x_0 \in [T_0,T]} \|w_h(x_0, \bullet)\|_{W^{1,\infty}(\mrn)}  \leq \rho, \;\  \sup_{x_0 \in [T_0,T]}\| (h D_{x^{\prime}})^k  w_h(x_0,\bullet)\|_{L^\infty(\mrn)} \leq \rho h, \\
&  \text{ and } \sup_{x_0 \in [T_0,T]} \| (h D_{x^{\prime}})^k  (D w_h)(x_0,\bullet)\|_{L^\infty(\mrn)} \leq \rho.
\end{split}
\eeq
 
 We will prove the following:
\begin{theorem}\label{approx}  Let $m\in \mn,$ $m> \novt+1$, and $M> m$, and let $\Omega= [T_0,T]_{x_0} \times \mrn_{x^{\prime}}$ and let  $f(x,u, w)\in C^\infty\bigl(\mr^{n+1}\times \mr \times \mr^{n+1}\bigr)$.   For any $\rho>0$, there exists $\eps_\rho>0$ and $r>0$ such that for any 
 $v_h\in \mca_{\rho,\eps_\rho}^{m+2}(\Omega)$ which satisfies
\beq\label{defvh}
\square v_h= f(x, v_h, D v_h)+ h^M G_h(x), \;\  G_h \in \mcb_{\rho,1}^m(\Omega),
\eeq
there exists a unique $u_h$,  $h\in (0, \eps_\rho]$,  such that $u_h-v_h \in h^M \mcb_{r,\eps_\rho}^m(\Omega)$ and  
\beq\label{semi}
\begin{split}
& \square u_h= f(x, u_h, D u_h),\\
& u\bigr|_{\{ x_0=T_0\}}= v_h\bigr|_{\{ x_0=T_0\}}, \;\   (\p_{x_0} u)\bigr|_{\{ x_0=T_0\}}= (\p_{x_0} v_h)\bigr|_{\{ x_0=T_0\}}.
\end{split}
\eeq
\end{theorem}

Before proving the theorem, we recall some techniques from \cite{Gue}. 

\subsection{Spaces of Distributions}  We recall norms of spaces of distributions introduced in \cite{Gue} needed to prove Theorem \ref{approx}. 
\begin{definition} If $u\in H^m(\mrn)$, $m\in \mn$, and $\mu>0$, we define
\beq\label{defmu}
\| u\|_{m,\mu} =\sum_{ |\alpha| \leq m} \mu^{m-|\alpha|}  \|D^\alpha u\|_{{}_{L^2(\mrn)}}.
\eeq
\end{definition}

We will also  need a particular version of the Gagliardo-Nirenberg inequalities, and we refer the reader to \cite{BreMir,Fio} for a proof.
 
 \begin{lemma}\label{Gagnir} If  $1\leq k \leq m\in \mn$, there exists a constant $C=C(k,m)$ such that
\beq\label{Gag1}
\sum_{|\alpha|\leq k} || D^\alpha u||_{L^{\frac{2m}{k}}} \leq C ||u||_{L^\infty}^{1-\frac{k}{m}} \left(\sum_{|\beta|\leq m} ||D^\beta u||_{L^2}\right)^{\frac{k}{m}}, \text{ provided } u \in L^\infty(\mr^{n}) \cap H^m(\mr^{n}).
\eeq
If one rescales $y=\frac{1}{\mu} x^\prime$  in \eqref{Gag1}, we obtain, with the same constant,
\beq\label{Gag2}
\sum_{|\alpha|\leq k} \mu^{k-|\alpha|}  || D^\alpha u||_{L^{\frac{2m}{k}}} \leq C ||u||_{L^\infty}^{1-\frac{k}{m}} \left(\sum_{|\beta|\leq m}\mu^{m-|\alpha|} ||D^\beta u||_{L^2}\right)^{\frac{k}{m}}, \text{ provided } u \in L^\infty(\mr^{n}) \cap H^m(\mr^{n}).
\eeq
\end{lemma}
 
 As a consequence of Lemma \ref{Gagnir} we have the following:
 \begin{lemma}\label{prodmu}  Let $u_j\in L^\infty(\mrn)\cap H^m(\mrn)$, $1\leq j \leq r$.  If $\alpha_j=(\alpha_{j_1}, \alpha_{j_2}, \ldots, \alpha_{j_n})$, $|\alpha_j|=\alpha_{j_1}+\alpha_{j_1}+\ldots \alpha_{j_n}$ and $\sigma=\sum_{j=1}^r |\alpha_j|  \leq m$, then there exists $C=C(m)$ such that
 \beq\label{Gag3}
 \mu^{m-\sigma} \bigl\|\prod_{j=1}^r D^{\alpha_j} u_j\bigr\|_{{}_{L^2(\mrn)}} \leq C \sum_{j=1}^r\biggl( \| u_j\|_{m,\mu} \prod_{i\not=j} \|u_i\|_{{}_{L^\infty(\mrn)}}\biggr).
 \eeq
 \end{lemma}
 \begin{proof}  If $\|u_j\|_{{}_{L^\infty}}=0$ for some $j$, this is obvious. Let us first assume $\mu=1$ and since
 \[
 \bigl\|\prod_{j=1}^r D^{\alpha_j} u_j\bigr\|_{{}_{L^2(\mrn)}}\leq \prod_{\alpha_j=0} \| u_j\|_{{}_{L^\infty}} \big\|\prod_{\alpha_j\not=0} D^{\alpha_j} u_j\bigr\|_{{}_{L^2(\mrn)}},
 \]
 we may assume $\alpha_j\not=0$ for all $j$.   Since $\sum_{j=1}^r \frac{|\alpha_j|}{2\sigma}=\ha$,  it follows from H\"older's inequality and \eqref{Gag1} that
 
 \[
 \begin{split}
   \bigl\|  \prod_{j=1}^r \bigl(D^{\alpha_j} u_{j}\bigr)\bigr\|_{{}_{L^2}} \leq &  \prod_{j=1}^r \bigl\| \bigl(D^{\alpha_j} u_{j}\bigr)\bigr\|_{L^{\frac{2\sigma }{|\alpha_j|}}}\leq 
C \prod_{j=1}^r \|u_{j}\|_{{}_{L^\infty}}^{1-\frac{|\alpha_j|}{\sigma}}  \|u_{j}\|_{{}_{{}_{H^{\sigma}}}}^{{}^{\frac{|\alpha_j|}{\sigma}}} \leq   \prod_{j=1}^r  C\|u_{j}\|_{{}_{L^\infty}} \biggl(\frac{ \|u_{j}\|_{{}_{{}_{H^{\sigma}}}}}{\|u_j\|_{{}_{L^\infty}}}\biggr)^{{}^{\frac{|\alpha_j|}{\sigma}}}.
 \end{split}
\]
 Now we recall that the convexity of the exponential function implies that
 \[
 \prod_{j=1}^r   \biggl(\frac{ \|u_{j}\|_{{}_{{}_{H^{\sigma}}}}}{\|u_{j}\|_{{}_{L^\infty}}}\biggr)^{{}^{\frac{|\alpha_j|}{\sigma}}} 
 \leq\sum_{j=1}^r \frac{|\alpha_j|}{\sigma} \biggl(\frac{ \|u_{j}\|_{{}_{{}_{H^{\sigma}}}}}{\|u_{j}\|_{{}_{L^\infty}}}\biggr),
 \]
 and since $\sigma\leq m,$ we obtain
 \[
 \bigl\|\prod_{j=1}^r D^{\alpha_j} u_j\bigr\|_{{}_{L^2(\mrn)}} \leq C \sum_{j=1}^r\biggl(\sum_{|\beta|\leq m} \|D^\beta u_j\|_{{}_{L^2}} \prod_{i\not=j} \|u_i\|_{{}_{L^\infty}}\biggr).
 \]
 
 If we rescale $y=\frac{x}{\mu}$ we obtain
 \[
\mu^{-\sigma} \bigl\|\prod_{j=1}^r D^{\alpha_j} u_j\bigr\|_{{}_{L^2(\mrn)}} \leq C \sum_{j=1}^r\biggl(\sum_{|\beta|\leq m} \mu^{-|\beta|} \|D^\beta u_j\|_{{}_{L^2}} \prod_{i\not=j} \|u_i\|_{{}_{L^\infty(\mrn)}}\biggr).
 \]
Multiply this inequality by $\mu^m$ and we obtain \eqref{Gag3}. This ends the proof of the Lemma.
 \end{proof}
 
 We will need the following:
 
 \begin{prop}\label{chainR1}  Let $u_j\in L^\infty(\mrn)\cap H^m(\mrn)$, $1\leq j \leq J$.  If  $F\in C^\infty(\mr^J)$, define 
 \[
 {\mathbf S}= \max \{  \|(D^\alpha F)(u_1, \ldots u_{{}_J})\|_{{}_{L^\infty}}:  1\leq j \leq J, |\alpha| \leq m\}.
 \] 
 If $m$ is a positive integer, there exists a constant $C= C(m,\|u_1\|_{{}_{L^\infty}}, \ldots \|u_J\|_{{}_{L^\infty}},  {\mathbf S})$ such that for all $\mu>0,$ 
 \beq\label{Gag4}
\|F(u_1, u_2, \ldots, u_{{}_J})\|_{m,\mu}\leq C \sum_{j=1}^J  \| u_j\|_{m,\mu}
 \eeq
 \end{prop}
 \begin{proof}
 The proof relies on a particular version of the multivariate Faa di Bruno formula  whose proof  will be left to the reader.  See \cite{ConSav} for the proof of the  complete formula, where the coefficients are computed.
\begin{lemma}\label{Faa}  Let $F \in C^\infty(\mr^J, \mr)$ and $ u_j \in C^\infty(\mr^{n}, \mr)$, $j=1,2,\ldots J$.  Let
\[
\begin{split}
& \overline P(J)= \{ \overline p=(p_1, p_2, \ldots, p_J):  \;\  p_j \in \mn, \;\ 1\leq p_j \leq m\}, \\
& \text{ and } \beta(j) =(\beta_1(j), \ldots \beta_J(j)) \text{ be a multi-index. }
\end{split}
\]
If  $\alpha=(\alpha_1, \alpha_2, \ldots, \alpha_{n})$,  there exist $C^\infty$ functions 
$G_{{}_{j,\beta(1), \ldots \beta(m),\ovp}} \in C^\infty(\mr^m,\mr)$, which are derivatives of  $F$ or order less than or equal to $|\alpha|$, such that
\beq\label{FDB}
\begin{split}
& D_{x^\prime}^\alpha F(u_1(x^\prime), u_2(x^\prime), \ldots, u_{j-1},u_J(x^\prime))= \\
& \sum_{\ovp \in P(J)} \sum_{j=1}^{|\alpha|} \sum_{\sum_{i=1}^j \beta(i)= \alpha} G_{{}_{j,\beta(1), \ldots \beta(J),\ovp}}(u_1(x^\prime), \ldots, u_{J}(x^\prime))  \prod_{i=1}^j \bigl(D_{x^\prime}^{\beta(i)} u_{p_i}\bigr).
\end{split}
\eeq
\end{lemma}

Since $\sum_{i=1}^j \beta(i)=\alpha$,  it follows from \eqref{Gag3} that there exists a constant $K=K(\|u_1\|_{{}_{L^\infty}}, \ldots, \|u_J\|_{{}_{L^\infty}})$ such that
\[
\mu^{m-|\alpha|} \biggl\|  \prod_{i=1}^j \bigl(D_{x^\prime}^{\beta(i)} u_{p_i}\bigr)\biggr\|_{{}_{L^2}} \leq  K \sum_{j=1}^J \| u_j\|_{m,\mu}.
\]
This ends the proof of the Proposition. 
\end{proof}
 
 We will also need the following more refined version of Proposition \ref{chainR1}
 \begin{lemma}\label{ChainR2}  Suppose $u_j  \in L^\infty(\mrn)\cap H^m(\mrn)$, $1\leq j \leq J$, and let $v_\mu(x')$  be a one-parameter family of functions defined for $\mu >1$ such that
\beq\label{estvmu}
\begin{split}
&  \|  v_\mu\|_{{}_{L^\infty}} \leq C_0, \;\  \| D_{x^\prime}^\alpha v_\mu\|_{{}_{L^\infty}}  \leq C_\alpha \mu^{|\alpha|-1}, \text{ if } |\alpha|\geq 1,  \text{ and let } {\mathbf C}= \max \{ C_0, C_\alpha, |\alpha|\leq m\}.
\end{split}
\eeq
   If $F\in C^\infty(\mr^{J+n+1})$, define 
 \[
 {\mathbf S}= \max \{ \bigl\|(D^\alpha F)(u_1, \ldots,  u_{J-1}, v_\mu, \nabla_{x'} v_\mu)\bigr\|_{{}_{L^\infty}}: |\alpha| \leq m\}.
 \]

 If $m$ is a positive integer, there exists a constant 
 
 \[
 K= C(m, \|u_j\|_{{}_{L^\infty}}, \ldots , \|u_J\|_{{}_{L^\infty}}, {\mathbf C} , {\mathbf S}), \;\ |\alpha|\leq m,
 \]
  such that for all $\mu>0,$ 
 \beq\label{Gag4N}
\|F(u_1,\ldots, u_J, v_\mu, \nabla_{x^\prime} v_\mu)\|_{m,\mu}\leq K \sum_{j=1}^{J} \| u_j\|_{m,\mu}.
\eeq
 \end{lemma}
\begin{proof}  We use \eqref{FDB} and denote $w=(w_1, \ldots, w_{J+n+1})=(u_1, \ldots, u_J, v_\mu,\nabla v_\mu)$, and therefore
\[
\prod_{i=1}^j D_{x^\prime}^{\beta(i)} w_{p_i}= \biggl(\prod_{ p_i\leq J } D_{x^\prime}^{\beta(i)} u_{p_i}
 \biggr) \biggl(D_{x^\prime}^{\beta(J+1)} v_{\mu} \biggr)\biggl( \prod_{p_i\geq J+2 } D_{x^\prime}^{\beta(i)} \p_{x^\prime_{p_i}}v_{\mu}\biggr).
 \]
 Say that $\sum_{r=1}^J \beta(i)=r_1$ and $\sum_{j= J+2}^{n+J+1} \beta(i)= r_2$, and so
 $|\alpha|=r_1+r_2+\beta(J+1)$, and therefore

 \[
\mu^{m-|\alpha|} \prod_{i=1}^j D_{x^\prime}^{\beta(i)} w_{p_i}=
\mu^{m-r_1}  \biggl(\prod_{ p_j\leq J }^j D_{x^\prime}^{\beta(i)} u_{p_i}
 \biggr) \biggl( \mu^{-\beta(J+1)} D_{x^\prime}^{\beta(J+1)} v_{\mu}\biggr) \biggl( \prod_{p_i\geq J+1 }^j \mu^{-\beta(i)} D_{x}^{\beta(i)} \p_{x_{p_i}}v_{\mu}\biggr),
 \]
  it follows from \eqref{estvmu} that
  \[
  \begin{split}
& \biggl\|\mu^{m-|\alpha|} \prod_{i=1}^j D_{x^\prime}^{\beta(i)} w_{p_i}\biggr\|_{{}_{L^2}} 
\leq  \mathbf{C}^{n+1} \mu^{m-r_1}\biggl\|\prod_{ p_j\leq J }^j D_{x^\prime}^{\beta(i)} u_{p_i}
 \biggr\|_{{}_{L^2}} 
 \end{split}
 \] 
 
 and therefore the result follows from \eqref{Gag3}.
 \end{proof}
We use the Lemma to prove 
\begin{prop}\label{ChainR3}  Let $u,v,w, D_{x^\prime} u, D_x w, \in L^{\infty}\bigl([0,T]_{x_0}, H^m(\mrn_{x^\prime})\bigr)\cap L^\infty(\mrn)$, where $m$ is a nonnegative integer and $\mu>1$.  Furthermore, let $v_\mu$ be a one parameter family of functions with $\mu>1$ such that  $v_\mu, \nabla v_\mu \in L^{\infty}\bigl([0,T]_{x_0}, H^m(\mrn_{x^\prime})\bigr)\cap L^\infty(\mrn)$ and 
\beq\label{estvmu1}
\begin{split}
 \sup_{x_0\in [0,T]} \| v_\mu\|_{{}_{L^\infty(\mrn)}}& \leq C_0,  \;\ \sup_{x_0\in [0,T]} \| D_{x'}^\alpha v_\mu \|_{{}_{L^\infty(\mrn)}} \leq C_\alpha \mu^{|\alpha|-1} \text{ and }  \\
& \sup_{x_0\in [0,T]}  \| D_{x'}^\alpha D_{x_0}v_\mu \|_{{}_{L^\infty(\mrn)}} \leq 
C_\alpha \mu^{|\alpha|}.
\end{split}
\eeq
Let ${\mathbf C}=\max\{ C_0, \;\ C_\alpha, \;\ |\alpha|\leq m\}$.  If  $F(x_0,x',r,\ga) \in C^\infty$, $x^\prime\in \mrn,$  $x_0$, $r\in \mr$ and $\ga \in \mr^{n+1}$, be compactly supported in $\{|x^\prime|\leq R\}$, let

  Let
  \[
 {\mathbf S}= \max \{ \bigl\|(D_{x'}^\alpha F)\bigl(x, u, v_\mu, w, Du, Dv_\mu,D w\bigr)\bigr\|_{{}_{L^\infty}}:   |\alpha| \leq m\}, 
 \] 
 If $m$ is a positive integer, there exists a constant 
 
 \beq\label{CM}
 K= K(m, \|u\|_{{}_{L^\infty}}, \|\nabla u\|_{{}_{L^\infty}},  \|w\|_{{}_{L^\infty}}, \|\nabla w\|_{{}_{L^\infty}}, {\mathbf C}, {\mathbf S}),
 \eeq
  such that for all $\mu>0,$  
\beq\label{Gag5}
\begin{split}
& \|F\bigl(x,u+v_\mu, D(u+v_\mu)\bigr)- F\bigl( x,w+v_\mu,D(w+v_\mu)\bigr)\|_{m,\mu} \leq \\
&K \biggl( \| u\|_{m+1,\mu}+ \| w\|_{m+1,\mu}+
\| D_t u\|_{m,\mu}+ \| D_tw\|_{m,\mu}  \biggr)\bigr( \|u-w\|_{{}_{L^\infty}}+ \|D(u-w)\|_{{}_{L^\infty}}\bigr)+
 \\
& K \bigl( \| u-w\|_{m+1,\mu} +  \| D_{x_0}(u-w)\|_{m,\mu}\bigr).
\end{split}
\eeq
\end{prop}
\begin{proof}  We write the difference as
\[
\begin{split}
 & F\bigl(x,u+v_\mu, D(u+v_\mu)\bigr)-  F\bigl(x,w+v_\mu,D(w+v_\mu)\bigr)=\\
 & \int_0^1 \frac{d}{dt} F(x, t(u+v_\mu)+(1-t)(w+v_\mu), t D(u+ v_\mu)+(1-t)D(w+v_\mu)) \ dt= \\
 & F_0\bigl(x, u, v_\mu, w, Du, Dv_\mu,D w\bigr) (u-w) + \sum_{j=1}^n F_j\bigl(x, u,v_\mu,w,Du,Dv_\mu,Dw\bigr) D_j(u-w),  \text{ where } \\
 & F_0\bigl(x, u, v_\mu, w, Du, Dv_\mu,D w\bigr) = \\
 & \int_0^1 (\p_r F)(x, t(u+v_\mu)+(1-t)(w+v_\mu), t D(u+ v_\mu)+(1-t)D(w+v_\mu)) \ dt \text{ and } \\
& F_j\bigl(x, u, v_\mu, w, Du, Dv_\mu,D w\bigr) = \\
& \int_0^1 (\p_{\ga_j} F)(x, t(u+v_\mu)+(1-t)(w+v_\mu), t D(u+ v_\mu)+(1-t)D(w+v_\mu)) \ dt.
 \end{split}
 \]
The product rule gives
\[
\begin{split}
& D^\alpha \biggl(  F_0(x, u,v_\mu,w,Du, Dv_\mu, Dw) (u-w)\biggr)= \\
& \sum_{\beta\leq \alpha } C_{\alpha,\beta}  D^\beta \biggl(F_0(x,u,v,w,Du,Dv,Dw)\biggr) \biggl(D^{\alpha-\beta} (u-w)\biggr),
\end{split}
\]
and it follows from \eqref{Gag3} that
\[
\begin{split}
& \mu^{m-|\alpha|} \biggl\| D^\alpha \biggl(  F_0(x, u, v_\mu,w, Du, Dv_\mu, Dw) (u-w)\biggr)\biggr\|_{L^2} \leq  \\
&  C \bigl\| F_0(x, u,v_\mu,w, Du, Dv_\mu , Dw) \bigr\|_{m,\mu}  \| u-w\|_{{}_{L^\infty}} + \\
&  C \bigl\| F_0(x, u, v_\mu,w, Du, D v_\mu,Dw) \bigr\|_{{}_{L^\infty}} \| u-w\|_{m,\mu}.
\end{split}
\]
But Lemma \ref{ChainR2} implies  that there exists a constant as in \eqref{CM} such that for $j=0,\ldots n,$
\beq\label{esfFj}
\begin{split}
& \bigl\| F_j(x, u,_\mu,w,Du,Dv_\mu,Dw) (u-w))\bigr\|_{m,\mu} \leq  \\
& C\bigl(\| u\|_{m+1,\mu}+\| D_t u\|_{m,\mu}+ \| w\|_{m+1,\mu}+  \| D_tw\|_{m+1,\mu}\bigr) \| u-w\|_{{}_{L^\infty}} + C \|u-w\|_{m,\mu}.
\end{split}
\eeq
Similarly, it follows from \eqref{Gag4} that
\[
\begin{split}
& \mu^{m-|\alpha|} \biggl\| D^\alpha \biggl(  F_j(x, u, v_\mu,w, Du, Dv_\mu, Dw) D_j(u-w)\biggr)\biggr\|_{L^2} \leq  \\
&  C \bigl\| F_j(x, u,v_\mu,w, Du, Dv_\mu , Dw) \bigr\|_{m,\mu}  \| D_j(u-w)\|_{{}_{L^\infty}} + \\
&  C \bigl\| F_j(x, u, v_\mu,w, Du, D v_\mu,Dw) \bigr\|_{{}_{L^\infty}} \| D_j(u-w)\|_{m,\mu}.
\end{split}
\]
and the estimate for these terms follows from \eqref{Gag4}.

\end{proof} 
\subsection{Energy Estimates} Next we prove weighted energy estimates needed in the proof of Theorem \ref{approx} and to agree with standard notation, we will use $x_0=t$ and $x=(t,x^{\prime}).$  As in \cite{Gue}, we equip the space $L^2([0,T]; H^m(\mrn))$, $m$ a non-negative integer, with the norm
\[
\| u\|_{m, \mu,\la}= \biggl[ \int_0^T \| u\|_{m,\mu}^2 e^{-2\la t} \ dt \biggr]^\ha, \;\ \;\ \mu>0, \;\ \la>0,
\]
where $ \| u\|_{m,\mu}$ was defined in \eqref{defmu}. 
\begin{prop} Let $u(t,x^{\prime})\in C_0^\infty(\mr^{n+1})$ and $m\in \mr_+$.  There exists a constant $C$ independent of $\la>0$, $t>0$, $\mu>0$  and $u$, such that
\beq\label{enerest0}
\begin{split}
&  e^{-\la T} \biggl(  \|u(T,\bullet)\|_{m+1,\mu}+\mu^{-1} \|\p_t u(T,\bullet)\|_{m,\mu} \biggr) + \sqrt{\la} \biggl( \|u\|_{m+1,\mu,\la}+ \frac{1}{\mu} \|\p_t u\|_{m,\mu,\la} \biggr) \leq  \\
& C\biggl( \|u(0,\bullet)\|_{m+1,\mu}+ \mu^{-1} \|\p_t u(0,\bullet)\|_{m+1,\mu}+ \la\mu^{-1} \|u(0,\bullet)\|_{m,\mu}+ \mu^{-1} \la^{-\ha} \|\square u\|_{m,\mu,\la}\biggr).
\end{split}
\eeq
\end{prop}
\begin{proof} We begin the proof with the following a priori  energy estimate

\begin{lemma} Let $u(t,x^{\prime})\in C_0^\infty(\mr^{n+1})$ and $m\in \mr_+$.  For $\la>0$ let
\[
E_{u,m,\la}(t)= \|\p_t u(t,\bullet)\|_{{}_{H^m(\mrn)}}+ \|u(t,\bullet)\|_{{}_{H^{m+1}(\mrn)}}+ \la \|u(t,\bullet)\|_{{}_{H^m(\mrn)}}.
\]
There exists a constant $C>0$ independent of $\la$, $t>0$,  and $u$, such that
\beq\label{enerest1}
 e^{-\la t} E_{u,m,\la}(t) +\sqrt{\la} \biggl[ \int_0^t e^{-2\la s} E_{u,m,\la}^2(s) \ ds\biggr]^\ha \leq  C \biggl( E_{u,m,\la}(0) +\frac{1}{\sqrt{\la}} \biggl[ \int_0^t e^{-2\la s} \| \square u(s, \bullet)\|_{{}_{H^m(\mrn)}}^2 \ ds\biggr]^\ha\biggr).
\eeq
\end{lemma}
\begin{proof}  Since $\square=\p_t^2-\Delta_{x'}$ commutes with derivatives in $x^{\prime},$ we just need to prove this for $m=0$.  If $\square u= F(t,x^{\prime})$,  and  $u= e^{\la t} v$, 
then $v$ satisfies
\[
\bigl( \square v+ 2\la \p_t v+ \la^2 v) =e^{-\la t} F(t,x^{\prime}).
\]
If we multiply this equation by $\p_t v+ \la v$, we obtain
\[
\begin{split}
& \ha \p_t\bigl[ (\p_t v)^2+ |\nabla_{x^{\prime}} v|^2 + 2 \la v \p_t v+ 3\la^2 v^2\bigr]+  \la \bigr[ (\p_t v)^2 + |\nabla_{x^{\prime}} v|^2+ \la^2 v^2\bigr]- \\
& \div_{x^{\prime}}( (\p_t v+\la v) \nabla_{x^{\prime}} v)= e^{-\la t} F(t,x^{\prime}) (\p_t v+\la v).
\end{split}
\]
We then integrate it $\mrn_{x^{\prime}}$, denote
\[
\begin{split}
& Z(v,t)=\ha \int_{\mrn} \bigl[ (\p_t v)^2+ |\nabla_{x^{\prime}} v|^2 + 2 \la v \p_t v+ 3\la^2 v^2\bigr]\ dx^{\prime}, \\
& W(v, t)= \int_{\mrn} \bigr[ (\p_t v)^2 + |\nabla_{x^{\prime}} v|^2+ \la^2 v^2\bigr] dx^{\prime}
\end{split}
\] 
and use the fact that $v$ is compactly supported, and we obtain
\[
\p_t Z(v,t) + \la W(v,t) =   2 e^{-\la t}\int_{\mrn} F(t,x^{\prime})(\p_t v+\la v) \ dx,
\]
and use  the Cauchy-Schwarz inequality we find that
\[
\p_t Z(v,t) + \la W(v,t) \leq   2 e^{-\la t} \|F(t,x^{\prime})\|_{{}_{L^2}}\sqrt{W(v,t)},
\]
and therefore 
\[
\p_t Z(v, t) + \la W(v, t) \leq   \frac{2}{\la} e^{-2\la t} \int_{\mrn} |F(t,x^{\prime})|^2 \ dx^{\prime} + \frac{\la}{2} W(v, t).
\]
We integrate this in $(0,t)$ and find
\[
Z(v, t) + \frac{\la}{2} \int_0^t W(v, s) \ ds \leq  Z(v,0)+ \frac{2}{\la} \int_0^t e^{-2\la s} \|F(s,\bullet)\|_{{}_{L^2(\mrn)}}^2 \ ds.
\]
Notice that $v =e^{-\la t} u$ and hence 
\[
\begin{split}
& (\p_t v)^2 + 2\la v \p_t v+ 3\la^2 v^2= e^{-2t} \bigl( (\p_t u)^2+ 2\la^2 u^2\bigr), \\
& (\p_t v)^2 + \la^2 v^2= e^{-2t} \bigl( (\p_t u)^2 -2\la u \p_t u+  2\la^2 u^2\bigr) \geq
e^{-2\la t} \bigl( \frac23(\p_t u)^2+ \frac14 \la^2 u^2\bigr).
\end{split}
\]
In the last inequality we used that  $-2\la u \p_t u\geq -\frac43 \la^2 -\frac34 (\p_t u)^2.$

This implies that there exists a constant $C>0$, independent of $u$ such that
\[
e^{-2\la t} W(u,t) + \la \int_0^t e^{-2\la s} W(u,s) \ ds \leq C\biggl( W(u,0) + \frac{1}{\la} \int_0^t e^{-2\la s} \|F(s,\bullet)\|_{{}_{L^2(\mrn)}}^2 \ ds\biggr).
\]

and this implies \eqref{enerest1} for $m=0$ and hence for all $m$.
\end{proof}

To prove \eqref{enerest0} we set $\tau=\mu t,$ $y=\mu x^{\prime}$, $w(\tau,y)= u(\frac{\tau}{\mu}, \frac{x^{\prime}}{\mu})$. If $\square u=F,$ we  obtain

\[
\bigl( \p_\tau^2-\Delta_y) w= \mu^{-2} G, \;\ G(\tau,y)= F(\frac{\tau}{\mu}, \frac{y}{\mu}).
\]

\eqref{enerest1} with $\la$ replaced by $\frac{\la}{\mu}$ applied to this equation gives
\beq\label{new-eest}
\begin{split}
&e^{-2\frac{\la}{\mu} \tau} E_{w,m,\frac{\la}{\mu}}(\tau) +\sqrt{ \frac{\la}{\mu}} \biggl[ \int_0^{\tau} e^{-2\frac{\la}{\mu} s} E_{w,m,\frac{\la}{\mu}}^2(s) \ ds\biggr]^\ha \leq \\
&C \biggl( E_{w,m,\frac{\la}{\mu}}(0) +\mu^{-2} \sqrt{ \frac{\mu}{\la}}  \biggl[ \int_0^{\tau} e^{-2\frac{\la}{\mu} s} \| G(s, \bullet)\|_{{}_{H^m(\mrn)}}^2 \ ds\biggr]^\ha\biggr)
\end{split}
\eeq

But notice that
\[
\begin{split}
& \| w(\tau,\bullet)\|_{m}= \sum_{|\alpha|\leq m} \| D_y^\alpha w\|_{{}_{L^2(\mrn)}} = \sum_{|\alpha|\leq m} \| D_y^\alpha \bigl( u(\frac{\tau}{\mu}, \frac{y}{\mu})\bigr)\|_{{}_{L^2(\mrn)}} 
=\\
&  \sum_{|\alpha|\leq m} \mu^{-|\alpha|} \| (D_y^\alpha u)(\frac{\tau}{\mu}, \frac{y}{\mu})\|_{{}_{L^2(\mrn)}} =
 \sum_{|\alpha|\leq m} \mu^{-|\alpha|+\novt } \| (D_{x^{\prime}}^\alpha u)(\frac{\tau}{\mu}, x^{\prime})\|_{{}_{L^2(\mrn)}}, 
  \end{split}
 \]
  and similarly
 \[
\begin{split}
&  \| D_\tau w(\tau,\bullet)\|_{m}= \mu^{-1+\novt} \sum_{|\alpha|\leq m} \mu^{-|\alpha|}\| (D_t D_{x^{\prime}}^\alpha u)(\frac{\tau}{\mu}, x^{\prime})\|_{{}_{L^2(\mrn)}}, 
 \end{split}
 \]
 and we also have
 \[
\begin{split}
   \| G(s, \bullet)\|_{{}_{H^m(\mrn)}}=\mu^{\novt}  \sum_{|\alpha|\leq m} \mu^{-|\alpha|} \| D_{x^{\prime}}^\alpha F(\frac{s}{\mu}, x^{\prime})\|_{{}_{L^2}}.
 \end{split}
 \]
 So we deduce that
 \[
 E_{w,m,\frac{\la}{\mu}}(\tau)=\mu^{\novt}\biggl( \mu^{-1} \| (\p_t u)(\frac{\tau}{\mu}, \bullet)\|_{m,\mu}+ \| u(\frac{\tau}{\mu}, \bullet)\|_{m+1,\mu} +\frac{\la}{\mu} \|u(\frac{\tau}{\mu}, \bullet)\|_{m,\mu}\biggr).
 \]
 Now substitute $\tau=\mu T$ and $s=\frac{\sigma}{\mu}$ in \eqref{new-eest}, we obtain \eqref{enerest0}.
 \end{proof}
 We define the spaces
 \[
 \mch_{1,m+1}(\Omega)= \{ u\in L^2([0,T]; H^{m+1}(\mrn)), \;\ \p_t u\in L^2([0,T]; H^{m}(\mrn))\}.
 \]
equipped with the norm
 \[
 \mcn_{m,\mu,\la}(u)= \biggl(\int_0^T e^{-2\la t} \|u(t,\bullet)\|_{m,\mu}^2 \ dt\biggr)^\ha +
 \biggl(\int_0^T e^{-2\la s} \|\p_tu(t,\bullet)\|_{m,\mu}^2 \ ds\biggr)^\ha,
 \]
 
 We will also need the following result, which is Lemma 2.8 of \cite{Gue}:
 \begin{lemma}\label{inc1} Let $m$ be an integer  and $m>\frac{n+1}{2}$. 
 There exists  $C(T,\la)>0$ such that if  $u \in \mch_{1,m+1}(\Omega)$, then
  \beq\label{inc2}
 \|u\|_{{}_{L^\infty}} +  \|D_x u\|_{{}_{L^\infty}} \leq \frac{C(T,\la)}{\mu^\oq} \mcn_{m,\mu,\la}(u).
 \eeq
\end{lemma}
We refer to \cite{Gue} for a proof.
 \subsection{The Proof of Theorem \ref{approx}}   Here we think of $h=\frac{\la}{\mu}$, with $\la$ fixed and we assume that there exists a one parameter family $v_\mu \in \mch_{1,m+1}(\Omega)$, $\mu>1$, such that

\beq\label{hypE}
\square v_\mu= f(x, v_\mu, Dv_\mu)  + \mu^{-M} \mce_\mu(x),
 \eeq
 such that
 \beq\label{hypE1}
\begin{split}
&  \mcn_{m+1,\mu,\la}(\mce_\mu)  \leq C \mu^{m}, \\
  \| v_\mu\|_{{}_{L^\infty(\Omega)}}+ & \| Dv_\mu\|_{{}_{L^\infty(\Omega)}} \leq R  \text{ and }  \mcn_{m,\mu,\la}(v_\mu) \leq C \mu^m, \\
\end{split}
 \eeq

Let $w_{\mu,j}$, $j=0,1,\ldots,$  be the sequence defined as follows
\beq\label{seq0}
\begin{split}
& \square w_{\mu,0}= - \mu^{-M} \mce_\mu(x), \\
& w_{\mu,0}(0,x^{\prime})=0, \;\ \p_t w_{\mu,0}(0,x^{\prime})=0, \\
\end{split}
\eeq

\beq\label{seq1}
\begin{split}
 \square w_{\mu,j}= & f\bigl(x, w_{\mu, j-1}+ v_\mu, D(w_{j-1}+v_\mu)\bigr)- f\bigl(x, v_\mu,D v_\nu\bigr) - \mu^{-M} \mce_\mu(x), \\
& w_{\mu,j}(0,x^{\prime})=0, \;\ \p_t w_{\mu,j}(0,x^{\prime})=0, \\
\end{split}
\eeq

We will show that there exists $\la_1>0$ such that for fixed $\la>\la_1$ one can find $\mu_1>0$ such that for all $\mu>\mu_1,$ there exists $w_\mu \in \mch_{1,m+1}$ such that  
\[
w_{\mu,j} \rightarrow w_\mu \in \mch_{1,m+1} \text{ and }   \mcn_{m,\mu,\la}(w_\mu)\leq  \mu^{M-m}.
\]
 Therefore, since $m> \novt+1,$ by continuity

\[
\square (w_\mu+ v_\mu)= f\bigl(x, w_\mu+v_\mu, D(w_\mu+v_\mu)\big),
\]
and therefore $u_\mu= w_\mu+v_\mu$ satisfies 
\[
\begin{split}
&\square u_\mu= f(x, u_\mu, D u_\mu), \\
u_\mu(0,x^{\prime})= & v_\mu(0,x^{\prime}),   \;\  \p_t u_\mu(0,x^{\prime})= \p_t v_\mu(0,x^{\prime}),
\end{split}
\]
and $u_\mu-v_\mu= w_\mu= \mu^{M-m} z_\mu,$ with $\mcn_{m,\mu,\la}(z_\mu) \leq C$.  This translates exactly into Theorem \ref{approx} by setting $\mu=\frac{\la}{h}$, with $\la>\la_0$ fixed, and $h<h_0= \frac{\la}{\mu_1}.$

 The proof relies on the following
 \begin{lemma}\label{contr}  Let $\Omega=[0,T]\times \mrn$, let $m$ be a positive integer with $m>\novt+1$ and let $M>m$. 
 Suppose there exists a family $v_{\mu}\in \mch_{1,m+1}(\Omega)$, satisfying \eqref{hypE} and \eqref{hypE1} and let $w_{\mu,j}$ be defined by \eqref{seq0} and \eqref{seq1},
  then there exist $\la_1>0$ such that  for fixed $\la>\la_1$,  there exists  $\mu_1>0$ and $C>0$ such that  for all $j=0,1,\ldots,$ 
\beq\label{cont}
\begin{split}
& \mcn_{m,\mu,\la}(w_{\mu,j}) \leq C \mu^{-M+m}, \text{ and } \\
N_{m,\mu,\la}\bigl( & w_{\mu,j+1}- w_{\mu,j}\bigr)  <  \ha N_{m,\mu,\la}\bigl( w_{\mu,j}- w_{\mu,j-1}\bigr).
\end{split}
\eeq
 \end{lemma}
\begin{proof} We know from \eqref{enerest0} that
\[
\|w_{\mu,0}\|_{m+1,\mu,\la} +\frac{1}{\mu} \|D_tw_{\mu,0}\|_{m,\mu,\la} \leq \frac{C\mu^{-M}}{\mu \la^{\tha}}  \|\mce_\mu(x)\|_{m,\mu,\la} \leq \frac{C\mu^{-M+m}}{\mu \la^{\tha}}.
\]
This implies that $w_{\mu,0}$ satisfies the first part of \eqref{cont}.  Suppose that $w_{\mu,k}$ satisfies the first part of  \eqref{cont} for all $k<j$, then it follows from \eqref{enerest0} that
\[
\begin{split}
& \|w_{\mu,j}\|_{m+1,\mu,\la} +\frac{1}{\mu} \|D_tw_{\mu,j}\|_{m,\mu,\la} \leq \\
\frac{C}{\la^{\tha} \mu} \bigl\| f\bigl(x, w_{\mu, j-1}& + v_\mu, D(w_{j-1}+v_\mu)\bigr)- f\bigl(x, v_\mu,D v_\mu\bigr)\bigr\|_{m,\mu,\la} + \frac{C \mu^{-M}}{\la^{\tha} \mu} \bigl\|\mce_\mu(x)\bigr\|_{m,\mu,\la}.
\end{split}
\]
We know from \eqref{inc2} that for $\mu$ large $\| v_\mu+ w_{\mu,k}\|_{{}_{L^\infty}}+ \| D(v_\mu+ w_{\mu,k})\|_{{}_{L^\infty}} \leq 2R$ for all $k<j$ and it follows from 
 Proposition \ref{ChainR3}  that
\[
\begin{split}
& \bigl\| f\bigl(x,w_{\mu,j-1}+v_\mu,D(w_{\mu,j-1}+v_\mu)\bigr)- f\bigl(x,w_{\mu,j-2}+v_\mu,D(w_{\mu,j-2}+ v_\mu)\bigr)\bigr\|_{m,\mu,\la} \leq \\
& H_{m,\mu,\la}\bigl(R, w_{\mu,j-1}, w_{\mu,j-2}\bigr)  \bigl( 1+ \| w_{\mu,j-1}-w_{\mu,j-2}\|_{{}_{L^\infty}}+  \| D_x (w_{\mu,j-1}-w_{\mu,j-2})\|_{{}_{L^\infty}}\bigr) \\
&\text{ where } H_{m,\mu,\la}(R, w_{\mu,j-1}, w_{\mu,j-2}) = C_1(R) \bigl(\|w_{\mu,j-1}-w_{\mu,j-2}\|_{m+1,\mu,\la} + \|D_t( w_{\mu,j-1}-w_{\mu,j-2})\|_{m,\mu,\la}\bigl).
\end{split}
\]
 But it follows from  \eqref{inc2} that
\beq\label{aux1}
\begin{split}
& \| w_{\mu,j-1}-w_{\mu,j-2}\|_{{}_{L^\infty}}+ \|D_x (w_{\mu,j-1}-w_{\mu,j-2})\|_{{}_{L^\infty}} \leq  \\
&\frac{C_2(\la)}{\mu^\oq} (\| w_{\mu,j-1}-w_{\mu,j-2}\|_{m+1,\mu,\la}+ \| D_t (w_{\mu,j-1}-w_{\mu,j-2})\|_{m,\mu,\la}),
\end{split}
\eeq
This shows that $w_{\mu,j}$ satisfies the first part of \eqref{cont}.  Now we need to verify the second part of \eqref{cont}. Since
\[
\begin{split}
 \square( w_{\mu,j}- w_{\mu,j-1})=&  f\bigl(x, w_{\mu,j-1}+ v_\mu, D(w_{\mu,j-1}+ v_\mu\bigr)- f\bigl(x, w_{\mu,j-2}+ v_\mu, D(w_{\mu,j-2}+ v_\mu\bigr), \\
& w_{\mu,j}(0,x^{\prime})=0 \text{ and }  \p_t w_{\mu,j}(0,x^{\prime})=0,
\end{split}
\]
it follows from \eqref{enerest0} that
\[
\begin{split}
& \|w_{\mu,j}-w_{\mu,j-1}\|_{m+1,\mu,\la} +\frac{1}{\mu} \|D_t(w_{\mu,j}-w_{\mu,j-1})\|_{m,\mu,\la} \leq \\
\frac{C}{\la^{\tha} \mu} \bigl\| f\bigl(x, w_{\mu, j-1}& + v_\mu, D(w_{j-1}+v_\mu)\bigr)- f\bigl(x, w_{j-2}+ v_\mu,D (w_{j-2}+v_\mu) \bigr)\bigr\|_{m,\mu,\la} .
\end{split}
\]
It follows from the first part of \eqref{cont} that there exists $\mu_0(R,\la,T)$ such that 
\beq\label{chla}
\| v_\mu+ w_{\mu,j}\|_{{}_{L^\infty}}+ \| D(v_\mu+ w_{\mu,j})\|_{{}_{L^\infty}} \leq 2R
\text{ for all } j, \text{ provided } \mu> \mu_0= \mu_0(R,\la,T).
\eeq
So we conclude from Proposition \eqref{ChainR3} that for $\mu>\mu_0,$
\[
\begin{split}
& \|w_{\mu,j}-w_{\mu,j-1}\|_{m+1,\mu,\la} +\frac{1}{\mu} \|D_t(w_{\mu,j}-w_{\mu,j-1})\|_{m,\mu,\la} \leq \\
&  \frac{C C_1(R)}{ \mu \la^\tha} \biggl(  \|w_{\mu,j-1}-w_{\mu,j-2}\|_{m+1,\mu,\la}+ \|D_t (w_{\mu,j-1}-w_{\mu,j-2})\|_{m,\mu,\la} \biggr) + \\
&\frac{C RC_1(R) C_2(\la)}{ \mu^{\frac54} \la^\tha} \biggl(  \|w_{\mu,j-1}-w_{\mu,j-2}\|_{m+1,\mu,\la}+ \|D_t (w_{\mu,j-1}-w_{\mu,j-2}) \|_{m,\mu,\la}\biggr)
\end{split}
\]
Now pick $\la_1$ such that $C C_1(R)\la^{-\tha} < \oq$, provided $\la > \la_1.$ Fixed $\la$ with   $\la>\la_1,$ pick $\mu_1$ such that $C R C_1(R) \la^{-\tha} C_2(\la)\mu^{-\del} < \oq$ for $\mu>\mu_1>\mu_0.$ This gives the second part of \eqref{cont}.

\end{proof}

\end{document}